\documentclass[12pt]{article}
\usepackage{amssymb, amsmath}
\usepackage{amsthm}
\usepackage{epsfig}

\usepackage{graphicx, amssymb, amsmath}
\usepackage{mathptmx}
\usepackage{latexsym}
\usepackage{longtable}
\usepackage{longtable}
\usepackage{rotating}
\usepackage{subfigure}

\newtheorem{theorem}{Theorem}

\newtheorem{corollary}[theorem]{Corollary}

\newtheorem{example}{Example}
\newtheorem{remark}{Remark}

\begin{document}

\begin{center}
{\Large A General Class of Weighted Rank Correlation Measures}
\end{center}

\begin{center}
{\Large M. Sanatgar$^{\mathrm{a}}$ and A. Dolati$^{\mathrm{b}}$
and M. Amini$^{\mathrm{c}}$}
\end{center}

\begin{center}
\textit{$^{\mathrm{a,b}}$Department of Statistics, College of
Mathematics, Yazd University, Yazd, Iran.\\[0pt]
\textrm{adolati@yazd.ac.ir}\\
$^{\mathrm{c}}$Department of Statistics, Ferdowsi University of Mashhad, Iran.\\[0pt]
\textrm{m-amini@um.ac.ir}}
\end{center}

\begin{abstract}

In this paper we propose a class of weighted rank correlation
coefficients extending the Spearman's rho. The proposed class
constructed by giving suitable weights to the distance between two
sets of ranks to place more emphasis on items having low rankings
than those have high rankings or vice versa. The asymptotic
distribution of the proposed measures and properties of the
parameters estimated by them are studied through the associated
copula. A simulation study is performed to compare the performance
of the proposed statistics for testing independence using
asymptotic relative efficiency calculations.
\end{abstract}

\section{Introduction}
Many situations exist in which $n$ objects are ranked by two
independent sources, where the interest is focused on agreement on
the top rankings and disagreements on items at the bottom of the
rankings, or vice versa. For example, every year a large number of
students apply for higher education. The graduate committee of the
university may like to choose the best candidates based on some
criteria such as GPA and the average of their grades in the major
courses that they passed during their bachelors level. In such
cases, to minimize the cost of interviewing all of the candidates,
a measure which gives more weighted for those who have higher
grades is required. Measures of rank correlation such as the
Spearman's rho and Kendall's tau generally give a value for the
overall agreement without giving explicit information about those
parts of a data set which are similar. This problem motivates the
definition of the \emph{Weighted Rank Correlation} (WRC) measures
which emphasize items having low rankings and de-emphasize those
having high rankings, or vice versa. Salama and Quade
(\cite{Salama1982},\cite{Quade1992}) first studied the WRC of two
sets of rankings, sensitive to agreements in the top rankings and
ignore disagreements on the rest items in a certain degree. For
application in the sensitivity analysis, Iman and Conover
\cite{Iman1987} proposed the top-down concordance coefficient
which centres on the agreement in the top rankings. Shieh
\cite{Shieh1998} studied a weighted version of the Kendall's tau
which could place more emphasis on items having low rankings than
those have high rankings or vice versa. Blest \cite{Blest2000}
introduced a rank correlation measure which gives more weights to
the top rankings. Pinto da Costa and Soares \cite{Pinto2005}
proposed a WRC measure that weights the distance between two ranks
using a linear function of those ranks, giving more importance to
higher ranks than lower ones. Maturi and Abdelfattah
\cite{Maturi2008} proposed a WRC measure with the different
weights to emphasize the agreement of the top rankings.
Coolen-Maturi \cite{Coolen2014} extended this index to the more
than two sets of rankings but again the focus was only on the
agreement on the top or bottom rankings. The behaviour of several
WRC measures derived from Spearman's rank correlation was
investigated by Dancelli et al. \cite{Dancelli2013}. Starting from
the formula of the Spearman's rank correlation measure, this paper
proposed a general class of WRC measures that weight the distance
between two sets of ranks. Two classes of weights, which are
polynomial functions of the ranks, are considered to place more
emphasis the items having low rankings than those have high
rankings or vice versa. The first one which extends the Blest's
rank correlation places more emphasis to the agreement on the top
rankings. The second one constructs a new class of WRC measures
and places more emphasis on the bottom ranks. The rest of the
paper is organized as follows. The proposed WRC measures are
introduced in Section 2. The weighted correlation coefficients
which are the population versions of the proposed WRC measures are
introduced in Section 3. The quantiles of the proposed WRC
measures for small samples and their asymptotic distributions for
large samples are presented in Section 4. A simulation study is
performed to compare the performance of the proposed statistics
for testing independence by using the asymptotic relative
efficiency and the empirical powers of the tests, in Section 5.
Finally, some discussions and possible extensions are given in
Section 6.
\section{The Proposed WRC measures}
Let $ (X_1,Y_1),...,(X_n,Y_n) $ be a random sample of size $ n $
from a continuous bivariate distribution and let
$(R_1,S_1),...,(R_n,S_n)$ denote the corresponding vectors of
ranks. The well-known Spearman's rank correlation measure is given
by
\begin{equation}\label{sp}
\rho_{n,s}=\frac{12}{n^3-n}\sum_{i=1}^{n}R_iS_i-\frac{3(n+1)}{n-1}.
\end{equation}
A drawback of the Spearman's rank correlation is that it generally
gives a value for the overall agreement of two sets of ranks
without giving explicit information about those parts of the sets
which are similar. For example consider 3 consumers A, B and C
that ranked 9 aspects of a product attributing '1' to the most
important aspect and '9' to the least important one. Their
rankings are given in Table 1. As we see, the top ranks of (A,B)
are more similar than those of (A,C) and the bottom ranks of (A,C)
are more similar than those of (A,B), but the Spearman's rank
correlation gives the same value 0.416 for two sets of rankings
(A,B) and (A,C).
\begin{table}\label{tab1}
\caption{Rankings of a product by 3 consumers}
\begin{center}
\begin{tabular}[l]{@{}|l|ccccccccc|}
\hline
A      & 1 & 2 & 3 & 4 & 5 & 6 & 7 & 8 & 9   \\
\hline
B      & {\bf 1} & {\bf 2} & \bf{3}  & 9&8&7&6&4&5  \\
\hline
C      & 5 &    6 & 4 & 3 & 2 & 1 &{\bf 7}&{\bf 8}& {\bf 9}\\

 \hline
\end{tabular}
\end{center}
\end{table}
For the cases where the differences in the top ranks would seem to
be more critical, Blest \cite{Blest2000} suggests that these
discrepancies should be emphasized. He proposed an alternative
measure of rank correlation that attaches more significance to the
early ranking of an initially given order. The Blest's index is
defined by
\begin{equation}\label{blest}
\gamma_n=\frac{2n+1}{n-1}-\frac{12}{n^2-1}\sum_{i=1}^{n}\left(1-\frac{R_i}{n+1}\right)^2S_i.
\end{equation}
The values of the Blest's index for two sets of rankings (A,B) and
(A,C) are given by 0.241 and 0.591, respectively. As we see, in
situations such as the rankings (A,C) where the bottom ranks
should be emphasized, the Blest's index is a suitable rank
correlation measure. In the following we develop a general theory
for weighted rank correlation measures by giving suitable weights
to the distance between two sets of ranks to place more emphasis
on items having low rankings than those have high rankings, or
vice versa. Let $D_i=S_i-R_i$, $i=1,...,n$. The most common form
of the Spearman's rank correlation coefficient between two sets of
rankings $R_1,...,R_n$ and $S_1,....,S_n$ is given by (Kendall,
\cite{Kendall1948})
\begin{equation}\label{spear}
\rho_{n,s}=1-\frac{2\sum\limits_{i=1}^n
D^2_i}{\max(\sum\limits_{i=1}^n D^2_i)},
\end{equation}
where $\max(\sum\limits_{i = 1}^nD^2_i)=(n^3-n)/3$ represents the
value of the summation when there is a perfect discordance between
rankings, that is, $S_i=n+1-R_i$, $i=1,...,n$. Throughout the rest
of the paper we assume, without loss of generality, that the
sample pairs are given in accordance with the increasing magnitude
of $X$ components, so that $R_i=i$, for $i=1,2,...,n$ and
$D_i=S_i-i$. According to Blest's idea \cite{Blest2000} if the set
of points $
(0,0),[(\sum_{i=1}^k(n+1-i),\sum_{i=1}^k(S_i-i),k=1,...,n] $
determined in the coordinate plane, the Spearman's rho is
normalized version of the sum of bars made the width of the given
points, i.e. $ \sum_{k=1}^n\sum_{i=1}^k(S_i-i) $ as a measure of
the disarray of originally ordered data, i.e.,
\begin{equation}\label{rblest}
\rho_{n,s}=1-\frac{2\sum_{k=1}^n\sum_{i=1}^k(S_i-i)}{\max(\sum_{k=1}^n\sum_{i=1}^k(S_i-i))}.
\end{equation}
By changing the order of summation, it is easy to see that $
\sum_{k=1}^n\sum_{i=1}^k(S_i-i)=\frac{1}{2}\sum_{k=1}^nD_i^2 $.
While \eqref{rblest} and \eqref{spear} are two different
representations of ordinary Spearman's rho, the Blest's index is
normalized version of the area which appears from connecting the
mentioned points to each other. By looking again to the Blest's
index, one can imagine that the bars $\eta_k=
\sum_{i=1}^k(S_i-i),k=1,...,n $ is assigned a certain weight, in
comparison to the Spearman's rho that does not give any weight to
the mentioned bar. Now we consider a general class of WRC measures
of the form
\begin{equation}\label{rw}
\nu_n=1-\frac{2\sum\limits_{i = 1}^n
w_i\eta_i}{\max(\sum\limits_{i = 1}^n w_i\eta_i)},
\end{equation}
where the positive constants $w_i$s are suitable weights. To
construct WRC measures which are sensitive to agreement on top
rankings (lower ranks), for $p=1,2,3...$ and $n>1$, we choose the
weights $w_i=(n+1-i)^p-(n-i)^p$. The class of WRC measures
constructed by (\ref{rw}) is then
\begin{equation}\label{rwl}
\nu_{n,p}^{(l)}=1 + \frac{2\sum\limits_{i = 1}^n
(i-S_i)(n+1-i)^p}{\sum\limits_{i = 1}^n (n+1-2i)(n+1-i)^p}.
\end{equation}
Alternatively, by choosing the weights $ w_i=i^p-(i-1)^p $, one
can obtain measures which are sensitive to agreement on bottom
rankings (upper ranks). In this case the class of WRC measures
(\ref{rw}) is simplified to
\begin{equation}\label{rwu}
\nu _{n,p}^{(u)} = 1 + \frac{2\sum\limits_{i = 1}^n
(i-S_i)(n^p-(i-1)^p)}{\sum\limits_{i = 1}^n
(n+1-2i)(n^p-(i-1)^p)}.
\end{equation}
Let $\kappa_{n,p}=\sum_{i=1}^n  i^p$. It is easily seen that
$$
\sum\limits_{i = 1}^n
(n+1-2i)(n+1-i)^p=2\kappa_{n,p+1}-(n+1)\kappa_{n,p},
$$
and
$$
\sum\limits_{i = 1}^n
(n+1-2i)(n^p-(i-1)^p)=2\kappa_{n-1,p+1}-(n-1)\kappa_{n-1,p}.
$$
The coefficients (\ref{rwl}) and (\ref{rwu}) could be rewritten in
terms of $\kappa_{n,p}$ as
\begin{equation}\label{rwl1}
\nu_{n,p}^{(l)}=\frac{(n+1)\kappa_{n,p}-2\sum_{i=1}^n
S_i(n+1-i)^p}{2\kappa_{n,p+1}-(n+1)\kappa_{n,p}},
\end{equation}
and
\begin{equation}\label{rwu1}
\nu_{n,p}^{(u)}=\frac{-(n+1)\kappa_{n-1,p}+2\sum_{i=1}^n
S_i(i-1)^p}{2\kappa_{n-1,p+1}-(n-1)\kappa_{n-1,p}}.
\end{equation}
Note that for $p=1$, both of $\nu_{n,p}^{(l)}$ and
$\nu_{n,p}^{(u)}$ reduce to the Spearman's rank correlation
coefficient (\ref{sp}). For $p=2$, the measure $\nu_{n,p}^{(l)}$
reduces to the Blest's rank correlation coefficient (\ref{blest}).
The coefficients $ \nu_{n,p}^{(l)} $ and $ \nu_{n,p}^{(u)} $ are
asymmetric WRC measures; that is, the correlation of $(X,Y)$ is
not the same as those of $(Y,X)$. One can obtain the symmetrized
version of (\ref{rwl}) as follows
\begin{eqnarray*}
\nu_{n,p}^{(s.l)}&=&\frac{\nu_{n,p}^{(l)}(X,Y)+\nu_{n,p}^{(l)}(Y,X)}{2}\\
&=&\frac{(n+1)\kappa_{n,p}-\sum_{i=1}^n[S_i(n+1-i)^p+i(n+1-S_i)^p]}{2\kappa_{n,p+1}-(n+1)\kappa_{n,p}}.
\end{eqnarray*}
Similarly the symmetrized version of (\ref{rwu}) is given by
\begin{eqnarray*}
\nu_{n,p}^{(s.u)} &=&\frac{-(n+1)\kappa_{n-1,p}+\sum_{i=1}^n
[S_i(i-1)^p+ i(S_i-1)^p]}{2\kappa_{n-1,p+1}-(n-1)\kappa_{n-1,p}}.
\end{eqnarray*}
For $ p=1 $ the WRC measures $ \nu_{n,p}^{(s.l)} $ and $
\nu_{n,p}^{(s.u)} $ are equal to the Spearman's rank correlation
(\ref{sp}). For $ p=2 $ the measure $\nu_{n,p}^{(s.l)}$ is the
symmetrized version of the Blest's index (\ref{blest}). Table 2
shows the values of $\nu_{n,p}^{(l)}$, $ \nu_{n,p}^{(u)}$,
$p=1,2,3,4,5$ and their symmetrized versions for the rankings
(A,B) and (A,C) given in Table 1. The result illustrates the
sensitivity of these indices to the agreement on top and bottom
rankings.
\begin{table}\label{T1}
\caption{\small Values of the Spearman's rho and WRC measures for
three sets of rankings in Table 1.}\label{sp.wrc}
\begin{center}
\renewcommand{\arraystretch}{2}
 \scalebox{.8}{\begin{tabular}[1]{@{}|c|cccc|cccc|}
 \hline
&~~~~~~~~($A,C$)&&&&~~~~~~~~($A,B$)&&&\\
\hline $p$&$ \nu^{(l)}_{n,p} $&$ \nu^{(u)}_{n,p} $&$
\nu^{(s.l)}_{n,p} $&$ \nu^{(s.u)}_{n,p} $
 &$ \nu^{(l)}_{n,p} $&$ \nu^{(u)}_{n,p} $&$ \nu^{(s.l)}_{n,p} $&$ \nu^{(s.u)}_{n,p} $ \\
\hline 1&0.433&0.433&0.433&0.433
&0.433&0.433&0.433&0.433\\
2&0.270&0.637&0.263&0.645
&0.596&0.229&0.603&0.220\\
3&0.155&0.768&0.140& 0.776
&0.720&0.112&0.728&0.094\\
4&0.081&0.851&0.057& 0.858
&0.808&0.045&0.815&0.016\\
5&0.033&{\bf 0.905}&0.001& {\bf 0.910}
&{\bf 0.869}&0.006&{\bf 0.875}&0.032\\
\hline\hline
\end{tabular}}
\end{center}
\end{table}
We note that $ \nu_{n,p}^{(l)} $, $ \nu_{n,p}^{(u)} $ and their
symmetrized versions take values in $ [ -1,1] $. In particular,
the value of these measures is equal to $ 1 $ when $ S_i = i $ (a
perfect positive dependency between two sets of ranks) and they
take $ -1 $ when $ S_i = n + 1-i $ (a perfect negative dependency
between two sets of ranks).
\section{Weighted correlation coefficients}
In this section we introduce the weighted correlation
coefficient's $\nu_p^{(l)}$ and $\nu_p^{(u)}$ and their
symmetrized versions $\nu_p^{(s.l)}$ and $\nu_p^{(s.u)}$ as the
population counterparts of the WRC measures $\nu_{n,p}^{(l)}$,
$\nu_{n,p}^{(u)}$, $\nu_{n,p}^{(s.l)}$ and $\nu_{n,p}^{(s.u)}$.
Each of these coefficients can be expressed as a linear functional
of the quantity $A(u,v)=C(u,v)-\Pi(u,v) $, where $ C $ is the
copula \cite{Sklar1959} associated with the pair $ (X,Y) $ and $
\Pi(u,v)=uv $ is the copula of independent random variables. For
$p=2,3,... $, we have
\begin{align}\label{pop}
\nu_p^{(l)}&=2(p+1)(p+2)\int_0^1\int_0^1(1-u)^{p-1}(C(u,v)-uv)dudv,\cr
\nu_p^{(u)}&=2(p+1)(p+2)\int_0^1\int_0^1u^{p-1}(C(u,v)-uv)dudv,\cr
\nu_p^{(s.u)} &=(p+1)(p+2)\int_0^1\int_0^1
(u^{p-1}+v^{p-1})(C(u,v)-uv)dudv,\cr
 \nu_p^{(s.l)} &=(p+1)(p+2)\int_0^1\int_0^1
((1-u)^{p-1}+(1-v)^{p-1}) (C(u,v)-uv)dudv.
\end{align}
Note that for $p=1$, all of these coefficients reduce to the
Spearman's rho given by
$$
\rho_s=12\int_{0}^{1}\int_{0}^{1}(C(u,v)-uv)dudv.
$$
For $p=2$, the coefficient $\nu^{(l)}_p$ reduces to the Blest's
correlation coefficient \cite{Genest2003} given by
\begin{equation*}
\gamma=24\int_{0}^{1}\int_{0}^{1}(1-u)C(u,v)dudv-2.
\end{equation*}
\begin{remark}
A probabilistic interpretation can be made for the weighted
correlation coefficients $\nu_p^{(l)}$ and $\nu_p^{(u)}$ and their
symmetrized versions $\nu_p^{(s.l)}$ and $\nu_p^{(s.u)}$. We
provide an illustration for $\nu_p^{(l)}$. For $ p = 1,2,... $
consider the cumulative distribution function
\begin{equation*}
F_p(u,v)=(1-(1-u)^p)v,\quad 0\leq u,v \leq 1.
\end{equation*}
Let $(U,V)$ be a random vector with the joint distribution
function $F_p$. For a copula $C$, the coefficient $\nu_p^{(l)}$
has the following representation
\begin{eqnarray*}\label{associationp}
\nu_p^{(l)}&=&\frac{2(p+1)(p+2)}{p}\int_{0}^{1}\int_{0}^{1}(C(u,v)-uv)dF_p(u,v)\\
&=&\frac{{{E_{F_p}}\left[ {C(U,V) - \Pi (U,V)}
\right]}}{{{E_{F_p}}\left[ {M(U,V) - \Pi (U,V)} \right]}},
\end{eqnarray*}
where $M(u,v)=\min(u,v)$ and $E_{F_{P}}$ denotes the expectation
with respect to $F_p$. Thus, the coefficient $\nu_p^{(l)}$ can be
considered as an average distance between the copula $C$ and the
independent copula $\Pi$, where the average is taken with respect
to the bivariate distribution function $F_p$. The proposed
weighted correlation coefficients could be seen as average
quadrant dependent (AQD) measures of association studied in
\cite{Behboodian2005}.
\end{remark}
For $\nu_p^{(l)}$, it is more convenient to use the following
alternative representation
\begin{equation}\label{altl}
\nu_p^{(l)}=\frac{2(p+1)(p+2)}{p}\int_{0}^{1}\int_{0}^{1}(1-u)^p(1-v)dC(u,v)-\frac{p+2}{p},
\end{equation}
which follows from the fact that
\begin{eqnarray*}
\int_{0}^{1}\int_{0}^{1}C(u,v)dF_p(u,v)&=&P(W\leq U,Z\leq V)\\
&=&P(U\geq W,V\geq
Z)\\&=&\int_{0}^{1}\int_{0}^{1}\bar{F}_p(u,v)dC(u,v),
\end{eqnarray*}
where $(W,Z)$ and $(U,V)$ are two independent pairs distributed as
the copula $C$ and the joint distribution $F_p$, respectively and
$\bar{F}_p(u,v)=P(U>u,V>v)=(1-u)^p(1-v)$, is the survival function
associated with $F_p$. An alternative representation for
$\nu_p^{(u)}$ is given by
\begin{equation}\label{altu}
\nu_p^{(u)}=\frac{2(p+1)(p+2)}{p}\int_{0}^{1}\int_{0}^{1}(1-u^p)(1-v)dC(u,v)-(p+2).
\end{equation}
In the following examples we provide the values of the weighted
correlation coefficient's $\nu_p^{(l)}$ and $\nu_p^{(u)}$ for some
copulas.
\begin{example}
Let $C_{\theta}$ be a member of the Cuadras-Aug\'e family of
copulas \cite{Nelsen2006} given by
\begin{equation}\label{cuadras}
C_{\theta}(u,v)=[\min(u,v)]^\theta [uv]^{1-\theta},\quad \theta
\in [0,1].
\end{equation}
This family of copula is positively ordered in $\theta \in [0,1]$,
that is, for $\theta_1\leq \theta_2$, we have that
$C_{\theta_1}(u,v)\leq C_{\theta_2}(u,v)$ for all $u,v\in [0,1]$.
This family of copulas has no lower tail dependence, whereas the
upper tail dependence parameter is given by $\lambda_U=\theta$
\cite{Nelsen2006}. For this family of copulas we have
\begin{equation}\label{quadu}
\nu^{(u)}_p=\frac{(p+2)(p^2+2p-3+\theta)}{p(p+3-\theta)},
\end{equation}
and
\begin{equation}\label{quadl}
\nu^{(l)}_p=\frac{\theta(p+2)\left(1-p(p+1)B(p,
4-\theta)\right)}{p(2-\theta)},
\end{equation}
where $B(a,b)=\int_{0}^{1}\int_{0}^{1}x^{a-1}(1-x)^{b-1}dx$, is
the beta function. For every $\theta\in[0,1]$, $\nu^{(u)}_p$ is
increasing in $p$ and $\nu^{(l)}_p$ is decreasing in $p$. For
$p=1,2,3,4,5,10 $, the values of $ \nu_p^{(l)}$ and $\nu_p^{(u)}$,
 as a function of $\theta $ is plotted in Figure
\ref{fig:1}. In particular for this family of copulas it follows
that for every $\theta\in[0,1]$ and $p=2,3,...$,
\begin{equation*}
\nu_{p}^{(l)}\leq\nu_{1}^{(l)}= \rho_s=\nu_1^{(u)} \leq
\nu_{p}^{(u)}.
\end{equation*}
\end{example}
\begin{figure}
\includegraphics[scale=0.52]{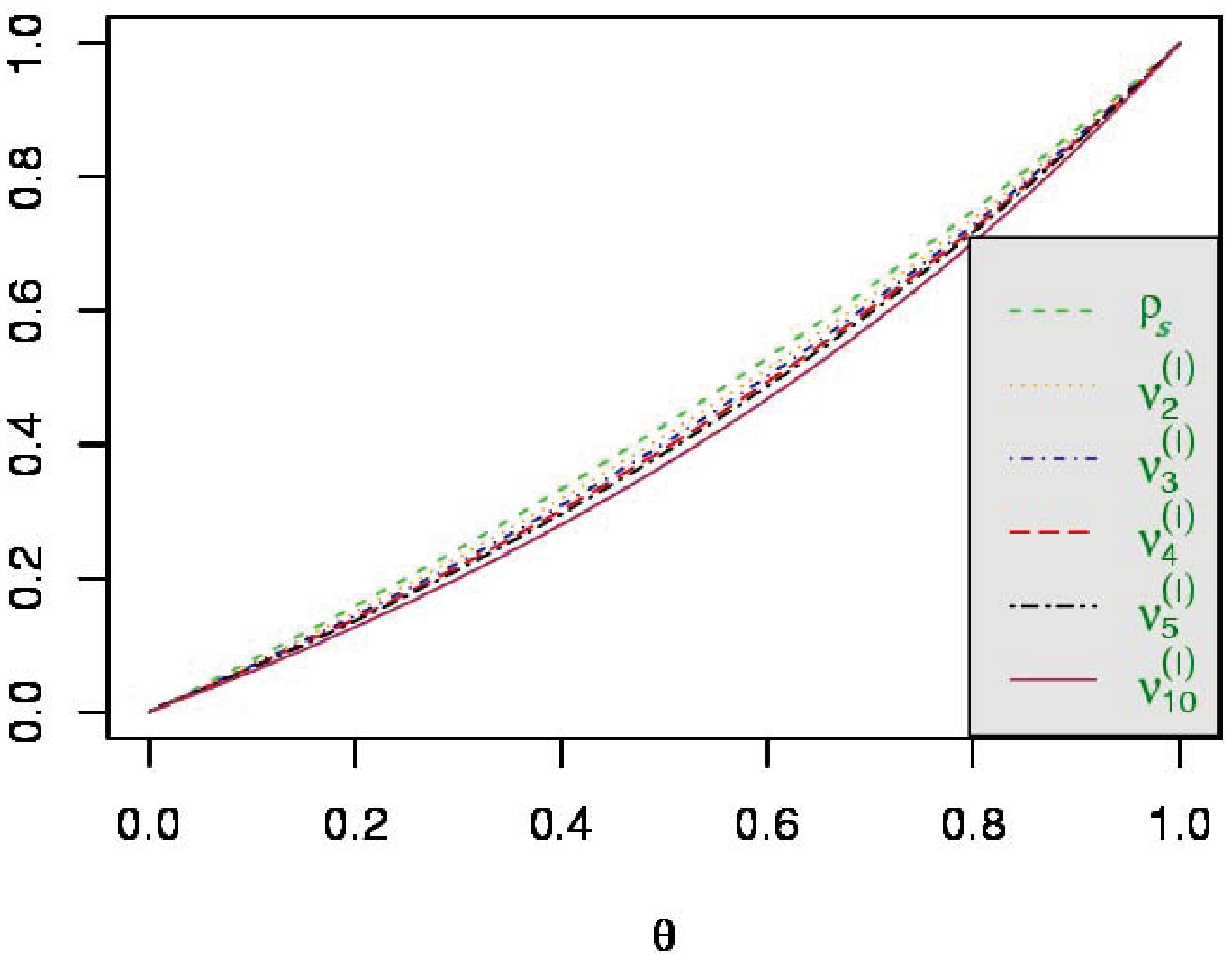}
\includegraphics[scale=0.48]{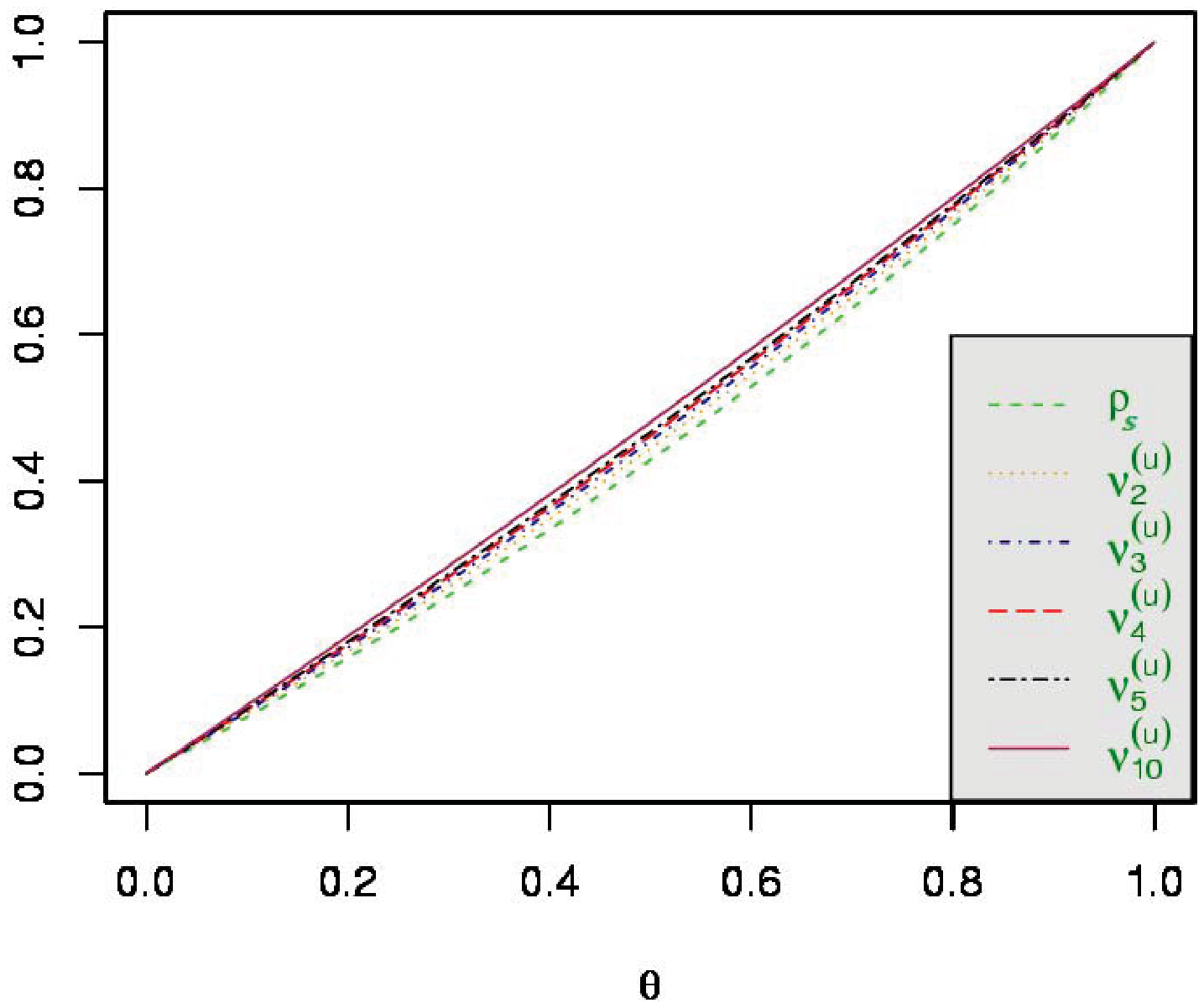}
\caption{ The values of $ \nu_p^{(l)}$ and $ \nu_p^{(u)}$,
$p=1,2,3,4,5,10 $, for Cuadras-Aug\'e family of
copulas.}\label{fig:1}
\end{figure}
\begin{example}
Let $C_{\theta}$ be a member of Raftery family of copulas
\cite{Nelsen2006} given by
\begin{equation*}
 C_{\theta} (u,v) = \min (u,v) + \frac{{1 - \theta }}{{1 + \theta }}{(uv)^{\frac{1}{{1 - \theta }}}}
 \left\{ {1 - {{[\max (u,v)]}^{ - \frac{{1 + \theta }}{{1 - \theta }}}}} \right\}, \quad\theta \in
 [0,1].
\end{equation*}
This family of copulas is also positively ordered in $ \theta \in
[0,1] $ and has no upper tail dependence, whereas the lower tail
dependence parameter is given by
$\lambda_L=\frac{2\theta}{\theta+1} $ \cite{Nelsen2006}. The
values of $ \nu_p^{(1)}$ and $ \nu_p^{(u)}$, $p=1,2,3,4,5,10 $, as
a function of $ \theta $ are plotted in Figure \ref{fig:2}. For
this family of copulas as we see for every $\theta \in [0,1]$ and
$p=2,3,...$,
\begin{equation*}
\nu_{p}^{(u)}\leq\nu_{1}^{(u)}=\rho_{s}=\nu_1^{(l)} \leq
\nu_{p}^{(l)}.
\end{equation*}
\end{example}
\begin{figure}
\includegraphics[scale=0.7]{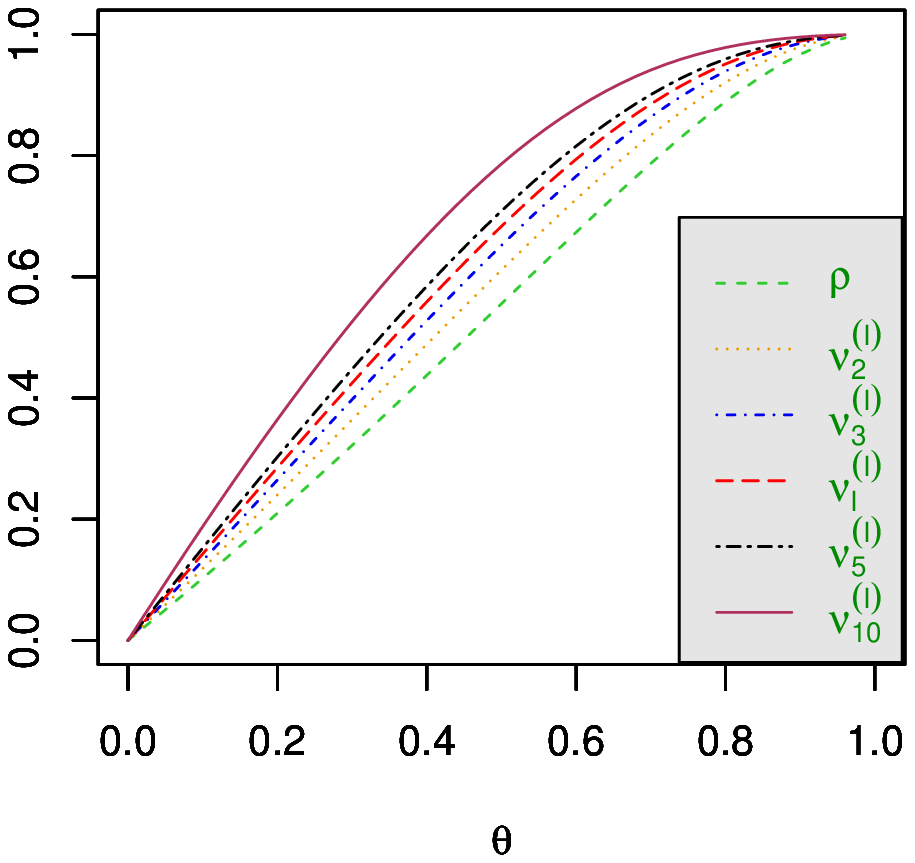}
\includegraphics[scale=0.7]{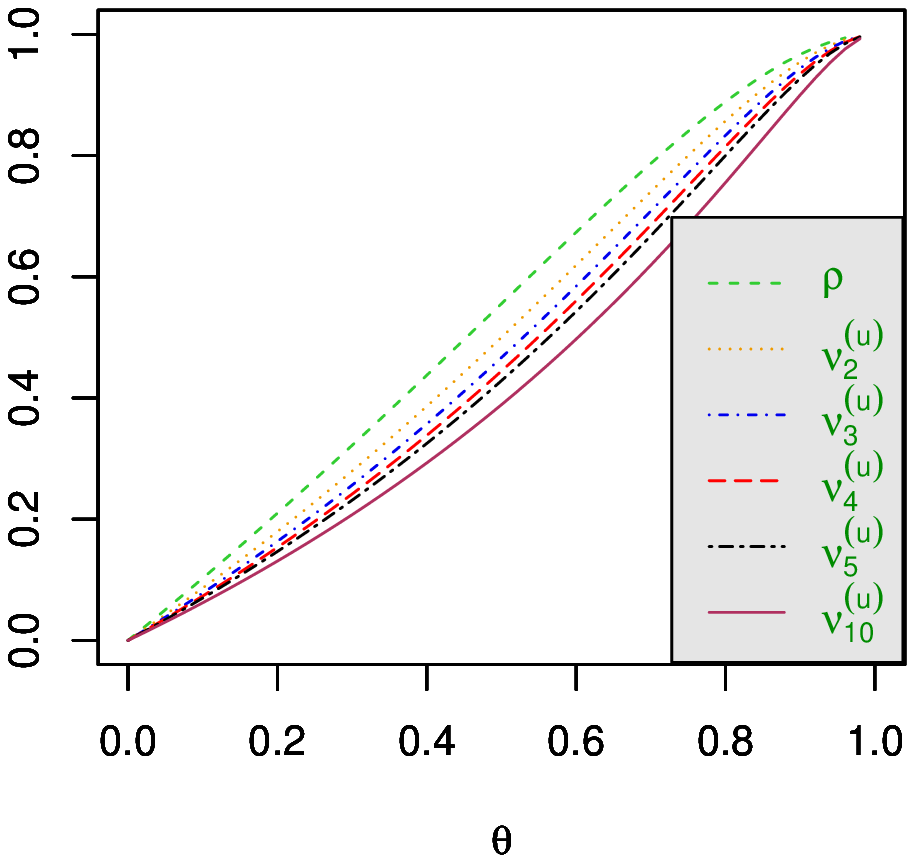}
\caption{The values of $ \nu_p^{(l)}$ and $ \nu_p^{(u)}$,
$p=1,2,3,4,5,10 $ for Raftery family of copulas.}\label{fig:2}
\end{figure}
\section{The Quantiles and Asymptotic Distributions}
The asymptotic behavior of the proposed WRC measures in general
can be studied by the standard results from the theory of
empirical processes \cite{Van1996}. Before that, we mention the
asymptotic formula for $ \kappa_{n,p}=\sum_{i=1}^n i^p$, that we
need in the sequel. By definition of the Riemann integral, it
holds that
\begin{equation}\label{riemman}
\frac{\kappa_{n,p}}{n^{p+1}}=\frac{1}{n}\sum_{i=1}^n\left(
\frac{i}{n}\right) ^p=\int_0^1 x^p dx+O(n^{-1}).
\end{equation}
Let $(X_1,Y_1),...,(X_n,Y_n)$ be a random sample of size $ n $
from a pair $(X,Y)$ of continuous random variables with the joint
distribution function $H$, marginal distribution functions $F$ and
$G$ and the associated copula $C$. Let $
(1,S_1),(2,S_2),...,(n,S_n) $ be the ranks of the rearranged
sample. It is known that (R$\ddot{{\text u}}$schendorf,
\cite{Ruschendorf1976}) the copula $C$ can be estimated by the
empirical copula defined for all $ u,v \in [0,1] $ by
\begin{equation*}
C_n(u,v) = \frac{1}{n}\sum\limits_{i = 1}^n I(\frac{i}{n +1} \le
u,\frac{S_i}{n + 1} \le v),
\end{equation*}
 where $I(A)$ denotes the indicator function of the set $A$.
The empirical versions of the weighted correlation coefficient's
$\nu_p^{(l)}$ and $\nu_p^{(u)}$ and their symmetrized versions
$\nu_p^{(s.l)}$ and $\nu_p^{(s.u)}$ defined by (\ref{pop}), could
be written in terms of the empirical copula $C_n$. By plugging
$C_n$ in (\ref{altl}), the empirical version of $ \nu_p^{(l)} $ is
of the form
\begin{equation*}
\tilde{\nu}_{n,p}^{(l)}=\frac{2(p+1)(p+2)}{p}\int_{0}^{1}\int_{0}^{1}(1-u)^p(1-v)dC_n(u,v)-\frac{p+2}{p}.
\end{equation*}
By using the representation (\ref{rwl1}) and the identity
\eqref{riemman}, straightforward calculations gives
\begin{eqnarray}\label{empl}
\tilde{\nu}_{n,p}^{(l)}\nonumber
&=&\frac{2(p+1)(p+2)}{np}\sum_{i=1}^n \left(1-\frac{i}{n+1}
\right)^p \left( 1-\frac{S_i}{n+1} \right)-\frac{p+2}{p}
\\\nonumber
&=&\frac{2(p+1)(p+2)\kappa_{n,p}}{np(n+1)^p}-\frac{2(p+1)(p+2)}{np(n+1)^{p+1}}\sum_{i=1}^n
S_i(n+1-i)^p-\frac{p+2}{p}\\\nonumber
&=&\frac{(p+1)(p+2)\kappa_{n,p}}{np(n+1)^p}+\frac{(p+1)(p+2)(2\kappa_{n,p+1}-(n+1)\kappa_{n,p})}{np(n+1)^{p+1}}
\nu_{n,p}^{(l)}-\frac{p+2}{p}\\&=&(1+O(n^{-1}))\nu_{n,p}^{(l)}+O(n^{-1}).
\end{eqnarray}
By using (\ref{altu}), a similar argument shows that the empirical
version of the coefficient $ \nu_p^{(u)}$ is given by
\begin{eqnarray}\label{empu}
 \tilde{\nu}_{n,p}^{(u)}&=&\nonumber
\frac{2(p+1)(p+2)}{np}\sum_{i=1}^n\left( 1-\left( \frac{i}{n+1}
\right)^p\right) \left( 1-\frac{S_i}{n+1} \right)-(p+2) \\
&=&(1+O(n^{-1}))\nu_{n,p}^{(u)}+O(n^{-1}).
\end{eqnarray}
In the following we provide the asymptotic distribution of the WRC
measures $\nu_{n,p}^{(l)}$, $\nu_{n,p}^{(u)}$ and their
symmetrized versions $\nu_{n,p}^{(s.l)}$, $\nu_{n,p}^{(s.u)}$. As
shown by Segers \cite{Segers2012}, $C_n$ converges weakly to $C$
as $n\rightarrow \infty$, whenever $C$ is regular, that is, the
partial derivatives $C_1(u,v)=\partial C(u,v)/\partial u$ and
$C_2(u,v)=\partial C(u,v)/\partial v$ exist everywhere on
$[0,1]^2$ and $C_1$ and $C_2$ are continuous on $(0,1)\times[0,1]$
and $[0,1]\times(0,1)$, respectively. Moreover, the empirical
copula process $\mathbb{C}_n=\sqrt{n}(C_n-C)$ converges weakly, as
$n\rightarrow \infty$, to a centered Gaussian process
$\hat{\mathbb{C}}$ on $[0,1]^2$, defined for all $u,v\in [0,1]$ by
\begin{equation}\label{mathbbcc}
\hat{\mathbb{C}}(u,v)=\mathbb{C}(u,v)-\frac{\partial }{{\partial
u}}C(u,v)\mathbb{C}(u,1)-\frac{\partial }{{\partial
v}}C(u,v)\mathbb{C}(1,v),
\end{equation}
where $\mathbb{C}(u,v)$ is Brownian bridge on $[0,1]^2 $ with the
covariance function
\begin{equation*}
E(\mathbb{C}(u,v)\mathbb{C}(s,t))=C(\min(u,s),\min(v,t))-C(u,v)C(s,t).
\end{equation*}
\begin{theorem}\label{theo.asy}
Suppose that $C$ is a regular copula. Then $ \sqrt n (\nu
_{n,p}^{(l)} - {\nu _p^{(l)}}) $, $\sqrt n (\nu_{n,p}^{(u)} - {\nu
_p^{(u)}}) $, $\sqrt n (\nu _{n,p}^{(s.u)} - {\nu_p^{(s.u)}}) $
and $\sqrt n (\nu _{n,p}^{(s.l)} - {\nu _p^{(s.l)}}) $ are
asymptotically centered normal with the asymptotic variances,
given by
\begin{small}
\begin{align}
(\sigma_p^{(l)})^2=&4(p+1)^2(p+2)^2\int_{[0,1]^4}(1-u)^{p-1}(1-s)^{p-1}E\left(\hat{\mathbb{C}}(u,v)\hat{\mathbb{C}}(s,t)\right)dudvdsdt,\\\nonumber
(\sigma_p^{(u)})^2=&4(p+1)^2(p+2)^2\int_{[0,1]^4}
u^{p-1}s^{p-1}E\left(\hat{\mathbb{C}}(u,v)\hat{\mathbb{C}}(s,t)\right)dudvdsdt,\\\nonumber
(\sigma_p^{(s.u)})^2=&(p+1)^2(p+2)^2\int_{[0,1]^4}
(u^{p-1}+v^{p-1})(s^{p-1}+t^{p-1})E\left(\hat{\mathbb{C}}(u,v)\hat{\mathbb{C}}(s,t)\right)dudvdsdt,\\\nonumber
(\sigma_p^{(s.l)})^2=&(p+1)^2(p+2)^2\int_{[0,1]^4}((u-1)^{p-1}+(v-1)^{p-1})((s-1)^{p-1}+(t-1)^{p-1})\cr
 & \,\,\,\,\,\,\,\,\,\,\,\,\,\,\,\,\,\,\,\,\,\,\,\,\,\,\,\,\,\,\,\,\,\,\,\,\,\,\,\,\,\,\,\,\,\,\,\,\,\,\,\,\,\,\,\,\,\,\,\,\,\,\,\, \times E
 \left(\hat{\mathbb{C}}(u,v)\hat{\mathbb{C}}(s,t)\right)dudvdsdt,
\end{align}
\end{small}
where $\hat{\mathbb{C}}$ is the Gaussian process defined by
\eqref{mathbbcc}.
\end{theorem}
\begin{proof}
We prove the result for $\sqrt n(\nu_{n,p}^{(l)} - {\nu _p^{(l)}})
$, similar arguments hold for $\sqrt n(\nu_{n,p}^{(u)} - {\nu
_p^{(u)}}) $, $\sqrt n (\nu _{n,p}^{(s.u)} -{\nu_p^{(s.u)}}) $ and
$\sqrt n (\nu _{n,p}^{(s.l)} - {\nu_p^{(s.l)}})$. From
(\ref{empl}) we have
\begin{align*}
\sqrt n & (\nu _{n,p}^{(l)} - {\nu _p}^{(l)}) = (1+O(n^{-1}))\sqrt
n (\tilde{\nu} _{n,p}^{(l)} - {\nu _p}^{(l)})+O(n^{-1/2}) \cr
&=2(1+O(n^{-1}))(p+1)(p+2)\int_0^1\int_0^1(1-u)^{p-1} \left[ \sqrt
n(C_n(u,v)-C(u,v)) \right] dudv+ O(n^{-1/2}).
\end{align*}
Since the integral on the right side is a linear and continuous
functional of the empirical copula process, then the left hand
side is asymptotically centered normal with asymptotic variance $
(\sigma_p^{(l)})^2 $, as stated in the theorem.
\end{proof}
\begin{corollary}
Assume the null hypothesis of independence i.e. $ C (u, v) = uv $.
Then $ \sqrt n  \nu_{n,p}^{(l)} $,
 $\sqrt{n} \nu_{n,p}^{(u)}$, $\sqrt{n} \nu_{n,p}^{(s.l)} $ and $ \sqrt{n} \nu_{n,p}^{(s.u)} $ are asymptotically centered normal
with the asymptotic standard deviations $
\sigma_p^{(l)}=\sigma_p^{(u)}=\frac{(p+2)}{\sqrt{3(2p+1)}} $ and
$\sigma_p^{(s.l)}=\sigma_p^{(s.u)}=\sqrt{\frac{p^2+10p+7}{6(2p+1)}}
$.
\end{corollary}
\begin{proof}
For $ C (u, v) = uv $ the covariance function of the limiting
Gaussian process $ \hat{\mathbb{C}} $ takes the form  $
E(\hat{\mathbb{C}}(u,v)\hat{\mathbb{C}}(s,t))=(\min(u,s)-us)(\min(v,t)-vt).
$ An application of Theorem 4.1 and a routine integration gives
the result.
\end{proof}
In order to use the proposed WRC measures for testing
independence, one needs to find their distributions or the
quantiles of the distribution under the hypothesis of
independence. The following result provides the expectation and
variance of $\nu_{n,p}^{(l)}$. Similar result could be found for
$\nu_{n,p}^{(u)}$, $\nu_{n,p}^{(s.l)}$ and $\nu_{n,p}^{(s.u)}$.
\begin{theorem}\label{asno} Under the hypothesis of independence
between two sets of ranks
\begin{equation*}
E(\nu_{n,p}^{(l)})=0,\quad
var(\nu_{n,p}^{(l)})=\frac{n(n+1)}{3}\frac{\kappa_{n,2p}-\frac{1}{n}(\kappa_{n,p})^2}{(2
\kappa_{n,p+1}-(n+1)\kappa_{n,p})^{2}}.
\end{equation*}
\end{theorem}
\begin{proof}
We note that the WRC measure $\nu_{n,p}^{(l)}$ by \eqref{rwl} can
be written as a linear combination of the linear rank statistic of
the form $ a_n+b_n\sum\limits_{i = 1}^n {a(i,S_i)} $, where
\begin{eqnarray*}
a_n&=&1+\left(2\sum\limits_{i = 1}^n i(n + 1 -
i)^p\right)\left(\sum\limits_{i = 1}^n (n + 1 - 2i)(n + 1 -
i)^p\right)^{-1}
\\
b_n&=&-2\left(\sum\limits_{i = 1}^n (n + 1 - 2i)(n + 1 -
i)^p\right)^{-1},
\end{eqnarray*}
and $ a(i,S_i)= S_i(n+1-i)^p$. The mean and the variance of the
quantity $S=\sum\limits_{i = 1}^n{a(i,S_i)}$ can be obtained, for
example, by using Theorem 1 in p. 57 in \cite{Sidak1999}. See,
also \cite{Hajek1969}.
\end{proof}
The exact and asymptotic variances of the normalized WRC measures
$\sqrt n \nu_{n,p}^{(l)},\sqrt n \nu_{n,p}^{(u)}$, $\sqrt n
\nu_{n,p}^{(s.l)}$ and $\sqrt n \nu_{n,p}^{(s.u)}$, under the
assumption of independence are provided in Table \ref{var}, for
$p=1,2,3,4,5$. According to the results of Table \ref{var} it
seems that the variance of the symmetric versions of WRC measures
are less than that of their asymmetric versions. They are more
appropriate for testing independence of two random variables.
\begin{table}\label{var}
\caption{\small Variance of the normalized WRC measures $\sqrt n
\nu_{n,p}^{(l)},\sqrt n \nu_{n,p}^{(u)}$, $\sqrt n
\nu_{n,p}^{(s.l)}$ and $\sqrt n \nu_{n,p}^{(s.u)}$, under the
assumption of independence, for $p=1,2,3,4,5$. }
\begin{center}
\small \scalebox{0.9}{
\begin{tabular}[l]{@{}|l|c|c|}
\hline
Index & The exact variance & Asymptotic variance   \\
\hline
$\sqrt n \nu_{n,1}^{(l)}$& $ \frac{n}{n-1} $ & $ 1 $ \\
\hline
$\sqrt n \nu_{n,2}^{(l)}$&$ \frac{n}{15}\frac { \left( 2\,n+1 \right)  \left( 8\,n+11 \right) }{ \left( n+1 \right) ^{2} \left( n-1 \right) } $&$\frac{16}{15}$\\
$\sqrt n \nu_{n,2}^{(u)}$&$ \frac{n}{15}\frac {16\,{n}^{2}-30\,n+11}{ \left( n-1 \right) ^{3}} $&$\frac{16}{15}$\\
$\sqrt n \nu_{n,2}^{(s.l)}$&$ \frac{n}{30}\frac {31\,{n}^{2}+60\,n+26}{ \left( n+1 \right) ^{2} \left( n-1 \right) } $&$\frac{31}{30}$\\
$\sqrt n \nu_{n,2}^{(s.u)}$&$ \frac{n}{30}\frac {31\,{n}^{2}-60\,n+26}{ \left( n-1 \right) ^{3}} $&$\frac{31}{30}$\\
\hline $\sqrt n \nu_{n,3}^{(l)}$& $ \frac{{25n}}{7}\frac{{27{n^4}
+ 84{n^3}
+ 69{n^2} - 8}}{{(n - 1){{(9{n^2} + 15n + 4)}^2}}} $ &  $\frac{25}{21}$ \\
$\sqrt n  \nu_{n,3}^{(u)} $& $ \frac{{25n}}{7}\frac{{27{n^4} -
84{n^3} +
69{n^2} - 8}}{{(n - 1){{(9{n^2} - 15n + 4)}^2}}} $& $\frac{25}{21}$\\
$\sqrt n  \nu_{n,3}^{(s.l)} $&$\frac{621n^5+1995n^4+1902n^3+420n^2-44n}{7(3n+1)^2(3n+4)^2(n-1)}$&$\frac{23}{21}$\\
$\sqrt n  \nu_{n,3}^{(s.u)} $&$ \frac{621n^5-1995n^4+1902n^3-420n^2-44n}{7(3n-4)^2(3n-1)^2(n-1)} $&$\frac{23}{21}$\\
\hline
$\sqrt n \nu_{n,4}^{(l)}$&$ \frac{n}{3}\frac{\left( 32\,{n}^{5}+119\,{n}^{4}+100\,{n}^{3}-65\,{n}^{2}-62\,n+31 \right)  \left( 2\,n+1 \right) }{ \left( 4\,{n}^{2}+5\,n-1 \right) ^{2} \left( n-1 \right)  \left( n+1 \right) ^{2}} $&$ \frac{4}{3} $\\
$\sqrt n \nu_{n,4}^{(u)}$&$\frac{n}{3} {\frac {64\,{n}^{6}-270\,{n}^{5}+319\,{n}^{4}+30\,{n}^{3}-189\,{n}^{2}+31}{ \left( 4\,{n}^{2}-5\,n-1 \right) ^{2} \left( n-1 \right) ^{3}}} $&$ \frac{4}{3} $\\
$\sqrt n \nu_{n,4}^{(s.l)}$&$ \frac{n}{6}\frac {112\,{n}^{6}+486\,{n}^{5}+658\,{n}^{4}+162\,{n}^{3}-195\,{n}^{2}-24\,n+34}{ \left( 4\,{n}^{2}+5\,n-1 \right) ^{2}\left( n-1 \right)  \left( n+1 \right) ^{2}} $&$ \frac{7}{6} $\\
$\sqrt n \nu_{n,4}^{(s.u)}$&$ \frac{n}{6}\frac {112\,{n}^{6}-486\,{n}^{5}+658\,{n}^{4}-162\,{n}^{3}-195\,{n}^{2}+24\,n+34}{ \left( 4\,{n}^{2}-5\,n-1 \right) ^{2} \left( n-1 \right) ^{3}} $&$ \frac{7}{6} $\\
\hline
$\sqrt n \nu_{n,5}^{(l)}$&$ {\frac {4900\,{n}^{9}+25872\,{n}^{8}+39396\,{n}^{7}-6468\,{n}^{6}-49539\,{n}^{5}+33467\,{n}^{3}-5880n}{ \left( 33\,n-33 \right)\left( 10\,{n}^{4}+28\,{n}^{3}+17\,{n}^{2}-7\,n-4 \right) ^{2}}} $&$ \frac{49}{33} $\\
$\sqrt n \nu_{n,5}^{(u)}$&$ {\frac {4900\,{n}^{9}-25872\,{n}^{8}+39396\,{n}^{7}+6468\,{n}^{6}-49539\,{n}^{5}+33467\,{n}^{3}-5880n}{ \left( 33\,n-33 \right)\left( 10\,{n}^{4}-28\,{n}^{3}+17\,{n}^{2}+7\,n-4 \right) ^{2}}} $&$ \frac{49}{33} $\\
$\sqrt n \nu_{n,5}^{(s.l)}$&$ \frac{n}{33}\frac {4100\,{n}^{8}+22176\,{n}^{7}+38244\,{n}^{6}+10164\,{n}^{5}-27789\,{n}^{4}-7623\,{n}^{3}+15298\,{n}^{2}+924\,n-2676}{ \left( n-1 \right)  \left( 10\,{n}^{4}+28\,{n}^{3}+17\,{n}^{2}-7\,n-4 \right) ^{2}} $&$ \frac{41}{33} $\\
$\sqrt n \nu_{n,5}^{(s.u)}$&$ \frac{n}{33}\frac {4100\,{n}^{8}-22176\,{n}^{7}+38244\,{n}^{6}-10164\,{n}^{5}-27789\,{n}^{4}+7623\,{n}^{3}+15298\,{n}^{2}-924\,n-2676}{ \left( n-1 \right)  \left( 10\,{n}^{4}-28\,{n}^{3}+17\,{n}^{2}+7\,n-4 \right) ^{2}} $&$ \frac{41}{33} $\\
\hline
\end{tabular}}
\end{center}
\end{table}
Under the assumption of independence, all $ n! $ orderings of a
set of rank $ (S_1,S_2,...,S_n) $ are equally likely to occur.
After calculating the value of the WRC measures between the two
rankings $ (1,2,...,n) $ and $ (S_1,S_2,...,S_n) $, their
quantiles are obtained by considering that the atom of the
discrete distribution is $ \frac{1}{n !} $. Here the $ r-
$quantiles in the data with $ x_i= $ the value of the specific WRC
measure in the $ i $-th paired samples, $i=1,2,..., n! $ are
calculated by using $(1-\gamma) x_{(j)}+\gamma x_{(j+1)},$ where $
\frac{j}{n} \leq r < \frac{j+1}{n} $, $x_{(j)}$ is the $j$th order
statistic and $\gamma=0.5$ if $nr=j$, and 1 otherwise. This
algorithm discussed in Hyndman and Fan \cite{Hyndman1996} as one
of the common methods for calculating quantiles for discontinuous
sample in statistical packages. The quantiles of normalized WRC
measures $\sqrt{n} \nu_{n,p}^{(l)}$, $\sqrt{n} \nu_{n,p}^{(u)}$
and their symmetrized versions $\sqrt{n} \nu_{n,p}^{(s.l)}$,
$\sqrt{n} \nu_{n,p}^{(s.u)}$ for $p=1,...,5 $ and $n=5,6,...,10$,
are given in Tables 4-5.

\begin{table}\label{ex.qu1}
\caption{\small The quantiles of normalized WRC measures $\sqrt{n}
\nu_{n,p}^{(l)}$ and $\sqrt{n} \nu_{n,p}^{(s.l)}$ for $p=1,...,5 $
and $n=5,6,...,10$.}
\centering \scriptsize
\begin{tabular}{|rr|lrrr|rrrr|}
\hline
   &  & $\sqrt n \nu_{n,p}^{(l)}$ &  &  &  & $\sqrt n \nu_{n,p}^{(s.l)}$  &   &  &  \\
  \hline
   $n$ & $p$ & $ 90\% $ & $ 95\% $ & $ 97.5\% $ & $ 99\% $ &  $ 90\% $ &  $ 95\% $ & $ 97.5\% $ & $ 99\% $ \\
 \hline
    n=5 & 1 & 1.565 & 1.789 & 2.012 & 2.012 & 1.565 & 1.789 & 2.012 & 2.012 \\
      & 2 & 1.565 & 1.845 & 2.012 & 2.124 & 1.565 & 1.826 & 2.012 & 2.124 \\
      &3 & 1.600 & 1.909 & 2.045 & 2.185 & 1.618 & 1.931 & 2.030 & 2.185 \\
      & 4 & 1.658 & 1.984 & 2.116 & 2.214 & 1.636 & 1.984 & 2.097 & 2.214 \\
       & 5 & 1.730 & 2.041 & 2.161 & 2.226 & 1.676 & 2.041 & 2.147 & 2.226 \\
      \hline
     n= 6 & 1 & 1.470 & 1.890 & 2.030 & 2.170 & 1.470 & 1.890 & 2.030 & 2.170 \\
       & 2 & 1.550 & 1.830 & 2.040 & 2.230 & 1.550 & 1.850 & 2.030 & 2.210 \\
       & 3 & 1.599 & 1.897 & 2.103 & 2.275 & 1.599 & 1.934 & 2.101 & 2.275 \\
       & 4 & 1.674 & 1.974 & 2.177 & 2.340 & 1.628 & 1.974 & 2.161 & 2.340 \\
       & 5 & 1.791 & 1.985 & 2.227 & 2.368 & 1.727 & 1.972 & 2.233 & 2.368 \\
      \hline
      n=7 & 1 & 1.465 & 1.795 & 1.984 & 2.268 & 1.465 & 1.795 & 1.984 & 2.268 \\
       & 2 & 1.500 & 1.831 & 2.079 & 2.268 & 1.488 & 1.819 & 2.079 & 2.268 \\
       & 3 & 1.569 & 1.891 & 2.129 & 2.323 & 1.551 & 1.877 & 2.112 & 2.323 \\
       & 4 & 1.656 & 1.980 & 2.196 & 2.378 & 1.620 & 1.954 & 2.204 & 2.366 \\
       & 5 & 1.725 & 2.078 & 2.253 & 2.436 & 1.707 & 2.019 & 2.248 & 2.424 \\
     \hline
      n=8 & 1 & 1.414 & 1.751 & 2.020 & 2.290 & 1.414 & 1.751 & 2.020 & 2.290 \\
       & 2 & 1.467 & 1.818 & 2.073 & 2.312 & 1.452 & 1.818 & 2.065 & 2.312 \\
       & 3 & 1.541 & 1.895 & 2.144 & 2.372 & 1.514 & 1.885 & 2.129 & 2.367 \\
       & 4 & 1.617 & 1.991 & 2.214 & 2.443 & 1.572 & 1.968 & 2.198 & 2.431 \\
       & 5 & 1.704 & 2.067 & 2.291 & 2.500 & 1.644 & 2.035 & 2.275 & 2.481 \\
     \hline
      n=9 & 1 & 1.400 & 1.750 & 2.050 & 2.300 & 1.400 & 1.750 & 2.050 & 2.300 \\
       & 2 & 1.445 & 1.800 & 2.070 & 2.330 & 1.430 & 1.795 & 2.070 & 2.330 \\
       & 3 & 1.523 & 1.884 & 2.152 & 2.404 & 1.488 & 1.861 & 2.137 & 2.395 \\
       & 4 & 1.611 & 1.973 & 2.238 & 2.479 & 1.548 & 1.938 & 2.212 & 2.466 \\
      & 5 & 1.694 & 2.061 & 2.314 & 2.550 & 1.612 & 2.019 & 2.284 & 2.532 \\
     \hline
     n=10 & 1 & 1.399 & 1.744 & 2.012 & 2.319 & 1.399 & 1.744 & 2.012 & 2.319 \\
      & 2 & 1.434 & 1.789 & 2.072 & 2.350 & 1.416 & 1.779 & 2.065 & 2.343 \\
      & 3 & 1.511 & 1.873 & 2.156 & 2.429 & 1.469 & 1.844 & 2.137 & 2.416 \\
      & 4 & 1.599 & 1.965 & 2.245 & 2.511 & 1.528 & 1.920 & 2.216 & 2.491 \\
      & 5 & 1.683 & 2.054 & 2.333 & 2.587 & 1.589 & 1.998 & 2.298 & 2.561 \\
\hline\hline
     $ n=\infty $  &   1 & 1.282 & 1.645 & 1.960 & 2.326 & 1.282 & 1.645 & 1.960 & 2.326 \\
       &2 & 1.324 & 1.699 & 2.024 & 2.403 & 1.303 & 1.672 & 1.992 & 2.365 \\
       &3 & 1.398 & 1.795 & 2.138 & 2.538 & 1.341 & 1.721 & 2.051 & 2.435 \\
       &4 & 1.480 & 1.899 & 2.263 & 2.686 & 1.384 & 1.777 & 2.117 & 2.513 \\
       &5 & 1.562 & 2.004 & 2.388 & 2.835 & 1.428 & 1.833 & 2.185 & 2.593 \\
\hline
\end{tabular}
\end{table}
\clearpage
\begin{table}\label{ex.qu1}
\caption{\small The quantiles of normalized WRC measures $
\sqrt{n} \nu_{n,p}^{(u)}$ and $ \sqrt{n} \nu_{n,p}^{(s.u)}$ for
$p=1,...,5 $ and $n=5,6,...,10$.} \centering \scriptsize
\begin{tabular}{|rr|lrrr|rrrr|}
\hline
   &  & $\sqrt n \nu_{n,p}^{(u)}$ &  &  &  & $\sqrt n \nu_{n,p}^{(s.u)}$  &   &  &  \\
  \hline
   $n$ & $p$ & $ 90\% $ & $ 95\% $ & $ 97.5\% $ & $ 99\% $ &  $ 90\% $ &  $ 95\% $ & $ 97.5\% $ & $ 99\% $ \\
 \hline
    n=5 & 1 & 1.565 & 1.789 & 2.012 & 2.012 & 1.565 & 1.789 & 2.012 & 2.012 \\
       & 2 & 1.593 & 1.873 & 2.012 & 2.180 & 1.565 & 1.901 & 2.012 & 2.180 \\
       & 3 & 1.663 & 1.982 & 2.120 & 2.222 & 1.633 & 1.982 & 2.098 & 2.222 \\
       & 4 & 1.749 & 2.053 & 2.176 & 2.232 & 1.690 & 2.053 & 2.162 & 2.232 \\
       & 5 & 1.865 & 2.102 & 2.205 & 2.235 & 1.814 & 2.102 & 2.197 & 2.235 \\
      \hline
      n=6 & 1 & 1.470 & 1.890 & 2.030 & 2.170 & 1.470 & 1.890 & 2.030 & 2.170 \\
       & 2 & 1.582 & 1.834 & 2.058 & 2.254 & 1.582 & 1.862 & 2.072 & 2.254 \\
       & 3 & 1.664 & 1.976 & 2.158 & 2.332 & 1.623 & 1.973 & 2.141 & 2.332 \\
       & 4 & 1.794 & 1.992 & 2.223 & 2.368 & 1.722 & 1.983 & 2.232 & 2.368 \\
       & 5 & 1.818 & 2.090 & 2.288 & 2.395 & 1.786 & 2.022 & 2.300 & 2.395 \\
      \hline
      n=7 & 1 & 1.465 & 1.795 & 1.984 & 2.268 & 1.465 & 1.795 & 1.984 & 2.268 \\
       & 2 & 1.528 & 1.858 & 2.095 & 2.284 & 1.528 & 1.858 & 2.087 & 2.284 \\
       & 3 & 1.611 & 1.940 & 2.182 & 2.365 & 1.579 & 1.936 & 2.187 & 2.357 \\
       & 4 & 1.715 & 2.070 & 2.249 & 2.427 & 1.676 & 2.012 & 2.240 & 2.410 \\
       & 5 & 1.818 & 2.137 & 2.316 & 2.487 & 1.763 & 2.078 & 2.332 & 2.461 \\
      \hline
      n=8 & 1 & 1.414 & 1.751 & 2.020 & 2.290 & 1.414 & 1.751 & 2.020 & 2.290 \\
       & 2 & 1.491 & 1.847 & 2.097 & 2.338 & 1.472 & 1.838 & 2.088 & 2.338 \\
       & 3 & 1.578 & 1.946 & 2.183 & 2.417 & 1.546 & 1.925 & 2.167 & 2.413 \\
       & 4 & 1.675 & 2.044 & 2.270 & 2.491 & 1.620 & 2.023 & 2.238 & 2.472 \\
       & 5 & 1.764 & 2.132 & 2.354 & 2.543 & 1.713 & 2.078 & 2.341 & 2.528 \\
      \hline
      n=9 & 1 & 1.400 & 1.750 & 2.050 & 2.300 & 1.400 & 1.750 & 2.050 & 2.300 \\
       & 2 & 1.469 & 1.825 & 2.094 & 2.356 & 1.444 & 1.812 & 2.094 & 2.350 \\
       & 3 & 1.563 & 1.925 & 2.195 & 2.442 & 1.516 & 1.897 & 2.176 & 2.434 \\
       & 4 & 1.662 & 2.027 & 2.289 & 2.527 & 1.587 & 1.990 & 2.260 & 2.507 \\
       & 5 & 1.745 & 2.123 & 2.374 & 2.603 & 1.660 & 2.073 & 2.338 & 2.584 \\
     \hline
     n=10 & 1 & 1.399 & 1.744 & 2.012 & 2.319 & 1.399 & 1.744 & 2.012 & 2.319 \\
      & 2 & 1.450 & 1.812 & 2.093 & 2.374 & 1.429 & 1.795 & 2.081 & 2.366 \\
      & 3 & 1.548 & 1.912 & 2.195 & 2.467 & 1.493 & 1.874 & 2.171 & 2.450 \\
      & 4 & 1.646 & 2.014 & 2.294 & 2.558 & 1.560 & 1.964 & 2.262 & 2.531 \\
      & 5 & 1.737 & 2.111 & 2.390 & 2.637 & 1.630 & 2.046 & 2.351 & 2.609 \\
\hline\hline
     $ n=\infty $  &   1 & 1.282 & 1.645 & 1.960 & 2.326 & 1.282 & 1.645 & 1.960 & 2.326 \\
       &2 & 1.324 & 1.699 & 2.024 & 2.403 & 1.303 & 1.672 & 1.992 & 2.365 \\
       &3 & 1.398 & 1.795 & 2.138 & 2.538 & 1.341 & 1.721 & 2.051 & 2.435 \\
       &4 & 1.480 & 1.899 & 2.263 & 2.686 & 1.384 & 1.777 & 2.117 & 2.513 \\
       &5 & 1.562 & 2.004 & 2.388 & 2.835 & 1.428 & 1.833 & 2.185 & 2.593 \\
    \hline
\end{tabular}
\end{table}

\section{ Efficiency of the Tests of Independence}
In this section we compare the Pitman asymptotic relative
efficiency (or, Pitman ARE) and the empirical power of tests of
independence constructed based on the proposed WRC measures.
\subsection{Pitman Efficiency}
Consider a parametric family $ \lbrace C_{\theta} \rbrace $ of
copulas with $ \theta=\theta_0 $ corresponding to the independence
case. Let $T_{1n}$ and $T_{2n}$ be two test statistics for testing
$H_0:\theta=\theta_0 $ versus $H_1:\theta> \theta_0 $ that reject
null hypothesis for large values of $ T_{1n} $ and $ T_{2n} $.
Suppose that $T_{1n}$ and $T_{2n}$ satisfy the regularity
conditions:

(1) There exist continuous functions $\mu_i(\theta)$ and
$\sigma_i(\theta)$, $ \theta > \theta_0, i=1,2 $ such that for all
sequences $\theta_n=\theta_0+\frac{h}{\sqrt{n}}$, $h>0$, it holds
\begin{equation*}
\mathop {\lim }\limits_{n \to \infty } {P_{{\theta _n}}}\left(
{\frac{{\sqrt n ({T_{in}} - \mu_i ({\theta _n}))}}{{\sigma_i
({\theta _n})}} < z} \right) = \Phi (z), z \in \mathrm{R},\quad
i=1,2,
\end{equation*}
where $\Phi(z)$ is the standard normal distribution function;

(2) The function $ \mu_i(\theta) $ is continuously differentiable
at $ \theta=\theta_0 $. and $ \mu_i^{'}(\theta_0)>0,
\sigma_i(\theta_0)>0, i=1,2. $

Under these conditions, the Pitman ARE of $T_{n1} $ relative to
$T_{n2}$ is equal to
\begin{equation*}
ARE({T_{1n}},{T_{2n}}) = \left[ {\frac{{\frac{\partial }{{\partial
\theta }}{\mu _1}{|_{\theta  = {\theta _0}}}}}{{\frac{\partial
}{{\partial \theta }}{\mu _2}{|_{\theta  = {\theta
_0}}}}}.\frac{{{\sigma _2}({\theta _0})}}{{{\sigma _1}({\theta
_0})}}} \right]^2.
\end{equation*}
For more detail see, \cite{niktin}. In the following we compare
the ARE of proposed WRC measures, relative to the Spearman's rho
for the Cuadras-Aug\'e family of copulas given by (\ref{cuadras}).
 \begin{example}
Suppose that the copula of $ (X, Y) $ be a member of the
Cuadras-Aug\'e family of copulas given by \eqref{cuadras} in
Example 3.1. Let $T_{1n}=\nu_{n,p}^{(l)}$ and $T_{2n}=\rho_{ns}$.
By Theorem 4.1, for the test statistics based on WRC measure
$\nu_{n,p}^{(l)}$, regularity conditions (1) and (2) are satisfied
with $ \theta_0 = 0, \mu_p(\theta)=\nu_p^{(l)} $ and
$\sigma_{p}^{l}$ given in Corollary 4.1 (for $ \theta_0 = 0 $). By
using (\ref{quadl}) and differentiation with respect to $\theta$
one gets
\begin{equation*}
\frac{d \nu_p^{(l)}}{d \theta}(0)=\frac{p+5}{2(p+3)}.
\end{equation*}
Since $\nu_1^{(l)}=\rho_s$ and $\nu_{n,1}^{(l)}=\rho_{n,s}$, then
from Corollary 4.1 we have
\begin{equation*}
ARE(\nu_{n,p}^{(l)},\rho_{ns})=\frac{4(p+5)^2(2p+1)}{3(p+2)^2(p+3)^2}.
\end{equation*}
Similarly, for $\nu_{n,p}^{(s.l)}, \nu_{n,p}^{(u)}$ and
$\nu_{n,p}^{(s.u)}$ we have
\begin{equation*}
ARE(\nu_{n,p}^{(s.l)},\rho_{n,s})=\frac{8(p+5)^2(2p+1)}{3(p+3)^2(p^2+10p+7)},
\end{equation*}
\begin{equation*}
ARE(\nu_{n,p}^{(u)},\rho_{n,s})=\frac{16(2p+1)}{3(p+3)^2},
\end{equation*}
and
\begin{equation*}
ARE(\nu_{n,p}^{(s.u)},\rho_{n,s})=\frac{32(p+2)^2(2p+1)}{3(p+3)^2(p^2+10p+7)}.
\end{equation*}
\end{example}
Table 8 shows the ARE of the test of independence based on WRC
measures compared to the test based on Spearman's rho for the
Clayton copula, as a family of copulas with lower tail dependence
and the Cuadras-Aug\'e family of copulas, as a family of copulas
with upper tail dependence. As we see, for the Clayton family of
copulas the measure $\nu_{n,11}^{(s.l)}$ and for the
Cuadras-Aug\'e family of copulas, the measure $\nu_{n,3}^{(s.u)}$
has the largest Pitman ARE. The same results is still true by
using Kendall's $\tau$ instead of Spearman's $\rho$ since $
ARE(T_{\tau},T_{\rho})=1 $.

\begin{table}\label{arecl.cu}
\caption{\small ARE of test of independence based on WRC measures
$\nu^{(l)}_{n,p},\nu^{(u)}_{n,p}, \nu^{(s.l)}_{n,p}$ and
$\nu^{(s.u)}_{n,p}$, for $p=1,2,...,13$, relative to the
Spearman's rho for the  Cuadras-Aug\'e and Clayton family of
copulas. } \small
\begin{center}
\renewcommand{\arraystretch}{2}
\scalebox{.8}{\begin{tabular}[1]{@{}|c|cccc|cccc|} \hline
&Cuadras-Aug\'e &&&& Clayton  &&&\\
\hline $p$&$ \nu^{(l)}_{n,p} $&$ \nu^{(u)}_{n,p} $&$
\nu^{(s.l)}_{n,p} $&$ \nu^{(s.u)}_{n,p} $
 &$ \nu^{(l)}_{n,p} $&$ \nu^{(u)}_{n,p} $&$ \nu^{(s.l)}_{n,p} $&$ \nu^{(s.u)}_{n,p}p $ \\
\hline 1&1&1&1&
1&1&1&1&1\\
2&0.816&1.066&0.843&1.101& 1.157&0.740&1.194&0.764
\\
3&0.663&1.037&0.721&{\bf 1.127}& 1.217&0.0.583&1.322& 0.634
\\
4&0.551&0.979&0.629&1.119& 1.235&0.480&1.411& 0.548
\\
5&0.467&0.916&0.558&1.095& 1.233&0.407&1.474& 0.486
\\
6&0.404&0.856&0.502&1.063& 1.221&0.353&1.518& 0.439
\\
7&0.355&0.800&0.457&1.028& 1.204&0.312&1.548& 0.401
\\
8&0.316&0.749&0.419&0.992& 1.184&0.279&1.569&0.370
\\
9&0.285&0.703&0.387& 0.956& 1.163&0.253&1.582&0.344
\\
10&0.258&0.662&0.360&0.922& 1.142&0.231&1.589&0.322
\\
11&0.237&0.625&0.336&0.888 & 1.119&0.213&{\bf 1.589}&0.302
\\
12&0.218&0.592&0.316& 0.857& 1.096&0.197&1.585&0.285
\\
13&0.202&0.562&0.297& 0.827& 1.028&0.183&1.580&0.270
\\
\hline\hline
\end{tabular}}
\end{center}
\end{table}
\begin{table}
 \caption{\small The power of tests of independence based on the Kendall's tau
($\tau_n$), Spearman's rho ($\rho_{n,s}$) the WRC measures
$\nu_{n,p}^{(s.l)}$, $\nu_{n,p}^{(s.u)}$, $p=2,3,4,5$, computed
from 50,000 samples of size 50 for Clayton copula with different
values of the parameter ($\theta$) and different level of
dependence in terms of Spearman's rho ($\rho_s$)}. \small
\begin{tabular}{|c|c|c|c|c|c|c|c|c|c|c|c|c|}
\hline
  $\theta$ & $\rho_{s}$  &$\nu_{n,5}^{(s.l)} $ & $\nu_{n,4}^{(s.l)}$ & $\nu_{n,3}^{(s.l)}$ & $\nu_{n,2}^{(s.l)}$ & $\nu_{n,5}^{(s.u)}$ & $\nu_{n,4}^{(s.u)}$ & $\nu_{n,3}^{(s.u)}$ & $\nu_{n,2}^{(s.u)}$ & $\rho_{n,s}$ & $\tau_n$ \\
 \hline
  0.000  &0.000 &     0.052 & 0.051 & 0.050 & 0.050 & 0.051 & 0.051 & 0.050 & 0.050 & 0.049 & 0.050 \\
  0.050  &0.036 &     0.095 & 0.092 & 0.090 & 0.087 & 0.071 & 0.072 & 0.075 & 0.077 & 0.082 & 0.082 \\
  0.110  &0.078 &     0.167 & 0.163 & 0.156 & 0.147 & 0.101 & 0.105 & 0.110 & 0.118 & 0.134 & 0.135 \\
  0.200  &0.136 &     0.309 & 0.301 & 0.289 & 0.271 & 0.158 & 0.168 & 0.183 & 0.204 & 0.242 & 0.244 \\
  0.350  &0.221 &     0.571 & 0.560 & 0.541 & 0.512 & 0.282 & 0.306 & 0.338 & 0.384 & 0.458 & 0.461 \\
  0.750  &0.397 &     0.937 & 0.934 & 0.928 & 0.915 & 0.641 & 0.688 & 0.742 & 0.806 & 0.880 & 0.881 \\
  1.800  &0.652 &     0.999 & 0.999 & 0.999 & 0.999 & 0.982 & 0.989 & 0.994 & 0.998 & 0.999 & 0.999 \\
  3.200  &0.800 &     1.000 & 1.000 & 1.000 & 1.000 & 0.999 & 0.999 & 1.000 & 1.000 & 1.000 & 1.000 \\
  5.600  &0.900 &     1.000 & 1.000 & 1.000 & 1.000 & 1.000 & 1.000 & 1.000 & 1.000 & 1.000 & 1.000 \\
 30.000  &0.993 &     1.000 & 1.000 & 1.000 & 1.000 & 1.000 & 1.000 & 1.000 & 1.000 & 1.000 & 1.000 \\
   \hline
\end{tabular}
\end{table}
\begin{center}

\begin{table}
\caption{\small The power of tests of independence based on the
Kendall's tau ($\tau_n$), Spearman's rho ($\rho_{n,s}$) the WRC
measures $\nu_{n,p}^{(s.l)}$, $\nu_{n,p}^{(s.u)}$, $p=2,3,4,5$,
computed from 50,000 samples of size 50 for Gumbel copula with
different values of the parameter ($\theta$) and different level
of dependence in terms of Spearman's rho ($\rho_s$)}. \small
\begin{tabular}{|c|c|c|c|c|c|c|c|c|c|c|c|}
\hline
   $\theta$ & $\rho_{s}$  &$\nu_{n,5}^{(s.l)} $ & $\nu_{n,4}^{(s.l)}$ & $\nu_{n,3}^{(s.l)}$ & $\nu_{n,2}^{(s.l)}$ & $\nu_{n,5}^{(s.u)}$ & $\nu_{n,4}^{(s.u)}$ & $\nu_{n,3}^{(s.u)}$ & $\nu_{n,2}^{(s.u)}$ & $\rho_{n,s}$ & $\tau_n$ \\
 \hline
  1.000   & 0.000  & 0.052          & 0.051 & 0.050 & 0.050 & 0.051 & 0.051 & 0.050 & 0.050 & 0.049 & 0.050 \\
  1.030   & 0.041  & 0.083          & 0.084 & 0.084 & 0.086 & 0.103 & 0.101 & 0.098 & 0.094 & 0.090 & 0.092 \\
  1.070   & 0.095  & 0.136          & 0.141 & 0.146 & 0.155 & 0.199 & 0.195 & 0.189 & 0.181 & 0.169 & 0.172 \\
  1.150   & 0.193  & 0.283          & 0.298 & 0.317 & 0.344 & 0.439 & 0.434 & 0.426 & 0.412 & 0.385 & 0.390 \\
  1.250   & 0.295  & 0.504          & 0.530 & 0.562 & 0.604 & 0.708 & 0.707 & 0.703 & 0.690 & 0.660 & 0.666 \\
  1.400   & 0.412  & 0.776          & 0.804 & 0.834 & 0.867 & 0.917 & 0.920 & 0.919 & 0.917 & 0.902 & 0.906 \\
  1.700   & 0.576  & 0.974          & 0.981 & 0.987 & 0.992 & 0.996 & 0.996 & 0.996 & 0.997 & 0.996 & 0.996 \\
  2.200   & 0.731  & 0.999          & 0.999 & 0.999 & 1.000 & 1.000 & 1.000 & 1.000 & 1.000 & 1.000 & 1.000 \\
  3.000   & 0.848  & 1.000          & 1.000 & 1.000 & 1.000 & 1.000 & 1.000 & 1.000 & 1.000 & 1.000 & 1.000 \\
  4.500   & 0.930  & 1.000          & 1.000 & 1.000 & 1.000 & 1.000 & 1.000 & 1.000 & 1.000 & 1.000 & 1.000 \\
\hline
\end{tabular}
\end{table}
\end{center}

\begin{table}
\caption{\small The power of tests of independence based on the
Kendall's tau ($\tau_n$), Spearman's rho ($\rho_{n,s}$) the WRC
measures $\nu_{n,p}^{(s.l)}$, $\nu_{n,p}^{(s.u)}$, $p=2,3,4,5$,
computed from 50,000 samples of size 50 for Normal copula with
different values of the parameter ($\theta$) and different level
of dependence in terms of Spearman's rho ($\rho_s$)}. \small
 \begin{tabular}{|c|c|c|c|c|c|c|c|c|c|c|c|c|}
 \hline
  $\theta$ & $\rho_{s}$  &$\nu_{n,5}^{(s.l)} $ & $\nu_{n,4}^{(s.l)}$ & $\nu_{n,3}^{(s.l)}$ & $\nu_{n,2}^{(s.l)}$ & $\nu_{n,5}^{(s.u)}$ & $\nu_{n,4}^{(s.u)}$ & $\nu_{n,3}^{(s.u)}$ & $\nu_{n,2}^{(s.u)}$ & $\rho_{n,s}$ & $\tau_n$
  \\
 \hline
 0.000   & 0.000 & 0.052 & 0.051 & 0.051 & 0.050 & 0.052 & 0.052 & 0.051 & 0.050 & 0.050 & 0.050 \\
 0.040   & 0.038 & 0.083 & 0.083 & 0.083 & 0.083 & 0.083 & 0.083 & 0.083 & 0.083 & 0.083 & 0.084 \\
  0.070  & 0.066 & 0.115 & 0.115 & 0.116 & 0.117 & 0.115 & 0.115 & 0.116 & 0.116 & 0.117 & 0.118 \\
  0.120  & 0.114 & 0.185 & 0.187 & 0.190 & 0.193 & 0.184 & 0.187 & 0.189 & 0.192 & 0.195 & 0.196 \\
  0.200  & 0.191 & 0.343 & 0.351 & 0.358 & 0.366 & 0.342 & 0.350 & 0.357 & 0.366 & 0.373 & 0.374 \\
  0.300  & 0.287 & 0.594 & 0.607 & 0.620 & 0.635 & 0.590 & 0.604 & 0.618 & 0.632 & 0.646 & 0.646 \\
   0.400 & 0.384 & 0.818 & 0.831 & 0.844 & 0.857 & 0.816 & 0.829 & 0.842 & 0.856 & 0.868 & 0.867 \\
   0.550 & 0.532 & 0.978 & 0.982 & 0.985 & 0.988 & 0.978 & 0.982 & 0.985 & 0.988 & 0.990 & 0.990 \\
   0.750 & 0.734 & 1.000 & 1.000 & 1.000 & 1.000 & 1.000 & 1.000 & 1.000 & 1.000 & 1.000 & 1.000 \\
   0.950 & 0.945 & 1.000 & 1.000 & 1.000 & 1.000 & 1.000 & 1.000 & 1.000 & 1.000 & 1.000 & 1.000 \\
   \hline
\end{tabular}
\end{table}
\subsection{Comparing the Power of Tests}
In the following we compare the power of tests based on the WRC
measures $\nu_{n,p}^{(s.l)}$ and $\nu_{n,p}^{(s.u)}$ for
$p=2,3,4,5 $ with the tests based on Kendall's tau and Spearman's
rho for testing independence against the positive quadrant
dependence \cite{Mutula}, i.e.,
\begin{equation*}
H_0:C(u,v)=uv \,\,\,\,\, \text{against} \,\,\,\,\, H_1:C(u,v) >
uv.
\end{equation*}
Monte Carlo simulations carry out for Gumbel, Clayton, Frank and
Normal copulas with various degrees of dependence, with the sample
of size $n=50$ at significance level $0.05$. The first column and
the second column of the Tables 7-10 indicate, the various values
of the copula parameter $ \theta $ and the value of the
corresponding spearman's
 $ \rho $ (as the level of dependence). Tables 7-10, show the power of tests that obtained under alternatives, defined by
Gumbel, Clayton, Frank and normal copulas. The Clayton copula has
lower tail dependence, the Gumbel copula has upper tail dependence
and the normal copula has neither. We see that for the Clayton's
family of copulas for all degree of dependence in terms of
Spearman's rho ($\rho_s)$, the test based on $\nu_{n,5}^{(s.l)}$
has the maximum power. For the Gumbel family of copulas, the test
based on $\nu_{n,5}^{(s.u)}$ has the maximum power. For the normal
family of copulas the behavior of all tests of independence are
the same. It seems that the members of the proposed class
$\nu_{n,p}^{(u)}$ defined by (\ref{rwu}) performs very well,
compared with the Kendall's tau and Spearman's rho, if there
exists a higher dependence in the upper tail. If there exists a
higher dependence in the lower tail, the members of the proposed
class $\nu_{n,p}^{(l)}$ defined by (\ref{rw}) has a better
performance.

\section{Discussion}

In this paper we have presented a class of weighted rank
correlation measures extending the Spearman's rank correlation
coefficient. The proposed class was constructed by giving suitable
weights to the distance between two sets of ranks to place more
emphasis on items having low rankings than those have high
rankings, or vice versa. The asymptotic distributions of the
proposed measures in general and under the null hypothesis of
independence are derived. We also carried out a simulation study
to compare the performance of the proposed measures with the
Spearman's and Kendall's rank correlation measures. Another line
of research is the extension of the result to the situations where
$n$ objects are ranked by $m>2$ independent sources and the
interest is focused on agreement on the bottom or top rankings.

\end{document}